\documentclass[reqno,a4paper, 11pt]{amsart}
\usepackage{amssymb}
\usepackage{array}
\usepackage{multirow}
\usepackage{slashbox}

\newtheorem{theorem}{Theorem}[section]
\newtheorem{proposition}[theorem]{Proposition}
\newtheorem{lemma}[theorem]{Lemma}
\newtheorem{corollary}[theorem]{Corollary}

\theoremstyle{definition}

\theoremstyle{remark}
\newtheorem{remark}[theorem]{Remark}

\numberwithin{equation}{section}

\begin{document}

\title{Ray class fields generated by torsion points of certain elliptic curves}

\author{Ja Kyung Koo}
\address{Department of Mathematical Sciences, KAIST}
\curraddr{Daejeon 373-1, Korea}
\email{jkkoo@math.kaist.ac.kr}
\thanks{}

\author{Dong Hwa Shin}
\address{Department of Mathematical Sciences, KAIST}
\curraddr{Daejeon 373-1, Korea}
\email{shakur01@kaist.ac.kr}
\thanks{}

\author{Dong Sung Yoon}
\address{Department of Mathematical Sciences, KAIST}
\curraddr{Daejeon 373-1, Korea} \email{yds1850@kaist.ac.kr}
\thanks{}


\keywords{elliptic curves, modular forms and functions, ray class
fields, Siegel functions.
\newline This research was supported by Basic Science Research Program through the National Research Foundation of Korea
funded by the Ministry of Education, Science and Technology
(2009-0063182).}

\maketitle

\begin{abstract}
We first normalize the derivative Weierstrass $\wp'$-function
appearing in Weierstrass equations which give rise to analytic
parametrizations of elliptic curves by the Dedekind $\eta$-function.
And, by making use of this normalization of $\wp'$ we associate
certain elliptic curve to a given imaginary quadratic field $K$ and
then generate an infinite family of ray class fields over $K$ by
adjoining to $K$ torsion points of such elliptic curve(Theorem
\ref{main}). We further construct some ray class invariants of
imaginary quadratic fields by utilizing singular values of the
normalization of $\wp'$, as the $y$-coordinate in the Weierstrass
equation of this elliptic curve(Theorem \ref{generator}, Corollary
\ref{Schertz}), which would be a partial result for the Lang-Schertz
conjecture of constructing ray class fields over $K$ by means of the
Siegel-Ramachandra invariant(\cite{Lang} p. 292, \cite{Schertz} p.
386).
\end{abstract}

\maketitle

\section{Introduction}

Let $K$ be an imaginary quadratic field with discriminant
$d_K\leq-7$ and $\mathcal{O}_K$ be its ring of integers. Let
$\theta$ be an element in the complex upper half plane
$\mathfrak{H}$ which generates $\mathcal{O}_K$, namely
$\mathcal{O}_K=[\theta,~1]$. For an elliptic curve $E$ (over
$\mathbb{C}$) with invariant $j(\mathcal{O}_K)=j(\theta)$ where $j$
is the elliptic modular function, there is an analytic
parametrization
\begin{equation}\label{Weier}
\varphi~:~\mathbb{C}/\mathcal{O}_K\stackrel{\sim}{\longrightarrow}
E\subset\mathbb{P}^2(\mathbb{C})~:~y^2=4x^3-g_2(\mathcal{O}_K)x-g_3(\mathcal{O}_K)
\end{equation}
where
$g_2(\mathcal{O}_K)=60\sum_{\omega\in\mathcal{O}_K\setminus\{0\}}\frac{1}{\omega^4}$
and $g_3(\mathcal{O}_K)=140\sum_{\omega\in
\mathcal{O}_K\setminus\{0\}}\frac{1}{\omega^6}$(\cite{Silverman}).
Let $h$ be the Weber function on $E$ defined by
\begin{equation*}
h(x,~y)=-2^73^5\frac{g_2(\mathcal{O}_K)g_3(\mathcal{O}_K)}{\Delta(\mathcal{O}_K)}x
\end{equation*}
where
$\Delta(\mathcal{O}_K)=g_2(\mathcal{O}_K)^3-27g_3(\mathcal{O}_K)^2$.
If $H$ and $K_{(N)}$ are the Hilbert class field and the ray class
field modulo $N\mathcal{O}_K$ of $K$ for each integer $N\geq2$,
respectively, we know from the main theorem of complex
multiplication that
\begin{equation*}
H=K\big(j(\mathcal{O}_K)\big)\quad\textrm{and}\quad
K_{(N)}=K\bigg(j(\mathcal{O}_K),~h\big(\varphi(\tfrac{1}{N})\big)\bigg)
\end{equation*}
(\cite{Lang} or \cite{Shimura}). Thus in order to describe a ray
class field $K_{(N)}$ we are to use only the $x$-coordinate of the
Weierstrass equation in (\ref{Weier}). However, we want to improve
in this paper the above result so that we are able to rewrite it as
\begin{equation*}
K_{(N)}=K\big(\varphi(\tfrac{1}{N})\big)=K\bigg(x\big(\varphi(\tfrac{1}{N})\big),~y\big(\varphi(\tfrac{1}{N})\big)\bigg)
\end{equation*}
by an appropriate modification of the curve in (\ref{Weier}).
\par
Ishida-Ishii showed in \cite{I-I} that for $N\geq7$ the function
field $\mathbb{C}\big(X_1(N)\big)$ of the modular curve
$X_1(N)=\Gamma_1(N)\backslash\mathfrak{H}^*$ can be generated by two
functions $X_2(\tau)^{\varepsilon_N N}$ and $X_3(\tau)^N$ where
$\mathfrak{H}^*=\mathfrak{H}\cup\mathbb{P}^1(\mathbb{Q})$,
$\Gamma_1(N)=\big\{\left(\begin{smallmatrix}a&b\\c&d\end{smallmatrix}\right)\in\mathrm{SL}_2(\mathbb{Z}):
\left(\begin{smallmatrix}a&b\\c&d\end{smallmatrix}\right)\equiv\left(\begin{smallmatrix}1&*\\0&1\end{smallmatrix}\right)\pmod{N}\big\}$
and
\begin{eqnarray*}
\varepsilon_N=\left\{\begin{array}{ll}1 & \textrm{if $N$ is odd}\\2
& \textrm{if $N$ is even}\end{array}\right.,\quad
X_2(\tau)=e^{(\frac{1}{N}-1)\frac{\pi
i}{2}}\prod_{t=0}^{N-1}\frac{\mathfrak{k}_{(\frac{2}{N},~\frac{t}{N})}(\tau)}
{\mathfrak{k}_{(\frac{1}{N},~\frac{t}{N})}(\tau)},\quad
X_3(\tau)=e^{(\frac{1}{N}-1)\pi
i}\prod_{t=0}^{N-1}\frac{\mathfrak{k}_{(\frac{3}{N},~\frac{t}{N})}(\tau)}
{\mathfrak{k}_{(\frac{1}{N},~\frac{t}{N})}(\tau)}
\end{eqnarray*}
as finite products of the Klein forms(see Section \ref{pre}).
 They further presented an algorithm to find a polynomial
$F_N(X,~Y)\in\mathbb{Z}[\zeta_N][X,~Y]$ with $\zeta_N=e^\frac{2\pi
i}{N}$ such that $F_N\big(X_2(\tau)^{\varepsilon_N
N},~X_3(\tau)^N\big)=0$, which can be viewed as an affine singular
model for the modular curve $X_1(N)$. And, for a fixed level $N$,
Hong-Koo(\cite{H-K}) pointed out that if
$P=\big(X_2(\theta)^{\varepsilon_N N},~X_3(\theta)^N\big)$ is a
nonsingular point on the curve defined by the equation
$F_N(X,~Y)=0$, then the ray class field $K_{(N)}$ is generated by
adjoining $P$ to $K$. But it leaves us certain inconvenience of
finding the polynomial $F_N(X,~Y)$ explicitly.
\par
In this paper we will develop this theme of \cite{H-K} from a
different point of view to overcome such inconvenience. First we
shall normalize the derivative Weierstrass $\wp'$-function by the
Dedekind $\eta$-function to be a modular function and then we
associate certain elliptic curve to a given imaginary quadratic
field $K$ with $d_K\leq-39$. Next, we will find an infinite family
of ray class fields $K_{(N)}$ generated by adjoining to $K$ certain
$N$-torsion points of such elliptic curve if $N\geq8$ and
$4~|~N$(Theorem \ref{main}).
\par
Furthermore, we shall show by adopting Schertz's
argument(\cite{Schertz}) that certain singular values of the
normalization of $\wp'$, as $y$-coordinates in the Weierstrass
equation of the above elliptic curve, give rise to ray class
invariants of $K_{(N)}$ over $K$ for some $N$, for example $N=p^n$
where $p$ is an odd prime which is inert or ramified in
$K/\mathbb{Q}$(Corollary \ref{Schertz}, Remark \ref{remark}). These
ray class invariants are, in practical use, simpler than those of
Ramachandra(\cite{Ramachandra}) consisting of too complicated
products of high powers of singular values of the Klein forms and
singular values of the $\Delta$-function. Here we note that Theorem
\ref{generator}, Remark \ref{power}, Corollary \ref{Schertz} and
Remark \ref{remark} give us partial results of the Lang-Schertz
conjecture concerning the Kronecker Jugendtraum over $K$.

\section{Modular forms and functions}\label{pre}

For a lattice $L$ in $\mathbb{C}$ the \textit{Weierstrass
$\wp$-function} is defined by
\begin{equation}\label{p-function}
\wp(z;~L)=\frac{1}{z^2}+\sum_{\omega\in
L\setminus\{0\}}\bigg(\frac{1}{(z-\omega)^2}-\frac{1}{\omega^2}\bigg)\qquad(z\in\mathbb{C}),
\end{equation}
and the \textit{Weierstrass $\sigma$-function} is defined by
\begin{equation*}
\sigma(z;~L)=z\prod_{\omega\in
L\setminus\{0\}}\bigg(1-\frac{z}{\omega}\bigg)e^{\frac{z}{\omega}+\frac{1}{2}(\frac{z}{\omega})^2}
\qquad(z\in\mathbb{C}).
\end{equation*}
Taking the logarithmic derivative we come up with the
\textit{Weierstrass $\zeta$-function}
\begin{equation*}
\zeta(z;~L)=\frac{\sigma'(z;~L)}{\sigma(z;~L)}=\frac{1}{z}+\sum_{\omega\in
L\setminus\{0\}}\bigg(\frac{1}{z-\omega}+\frac{1}{\omega}+\frac{z}{\omega^2}\bigg)\qquad(z\in\mathbb{C}).
\end{equation*}
Then, differentiating the function $\zeta(z+\omega;~L)-\zeta(z;~L)$
for $\omega\in L$ results in $0$, because $\zeta'(z;~L)=-\wp(z;~L)$
and the $\wp$-function is periodic with respect to $L$. Hence there
is a constant $\eta(\omega;~L)$ such that
$\zeta(z+\omega;~L)=\zeta(z;~L)+\eta(\omega;~L)$.
\par
For a pair $(r_1,~r_2)\in\mathbb{Q}^2\setminus\mathbb{Z}^2$ we
define the \textit{Klein form} as
\begin{equation*}
\mathfrak{k}_{(r_1,~r_2)}(\tau)=e^{-\frac{1}{2}(r_1\eta_1+r_2\eta_2)(r_1\tau+r_2)}\sigma(r_1\tau+r_2;~[\tau,~1])\qquad(\tau\in\mathbb{C})
\end{equation*}
where $\eta_1=\eta(\tau;~[\tau,~1])$ and
$\eta_2=\eta(1;~[\tau,~1])$. And we define the \textit{Siegel
function} by
\begin{equation*}
g_{(r_1,~r_2)}(\tau)=\mathfrak{k}_{(r_1,~r_2)}(\tau)\eta^2(\tau)\qquad(\tau\in\mathfrak{H})
\end{equation*} where $\eta$ is the \textit{Dedekind $\eta$-function}
satisfying
\begin{equation}\label{Fouriereta}
\eta(\tau)=\sqrt{2\pi}\zeta_{8}q_\tau^{\frac{1}{24}}\prod_{n=1}^\infty
(1-q_\tau^n)\qquad(q_\tau=e^{2\pi i\tau},~\tau\in\mathfrak{H}).
\end{equation}
If we let $\mathbf{B}_2(X)=X^2-X+\frac{1}{6}$ be the second
Bernoulli polynomial, then from the $q_\tau$-product formula of the
Weierstrass $\sigma$-function(\cite{Lang} Chapter 18 Theorem 4) and
(\ref{Fouriereta}) we get the following  Fourier expansion formula
\begin{eqnarray}
\label{FourierSiegel}g_{(r_1,~r_2)}(\tau)=-q_\tau^{\frac{1}{2}\mathbf{B}_2(r_1)}e^{\pi
ir_2(r_1-1)}(1-q_z)\prod_{n=1}^{\infty}(1-q_\tau^nq_z)(1-q_\tau^nq_z^{-1})
\end{eqnarray}
where $q_z=e^{2\pi iz}$ with $z=r_1\tau+r_2$. Here we note that
$\eta(\tau)$ and $g_{(r_1,~r_2)}(\tau)$ have no zeros and poles on
$\mathfrak{H}$ due to (\ref{Fouriereta}) and (\ref{FourierSiegel}).
And, we have the order formula
\begin{equation}\label{order}
\mathrm{ord}_{q_\tau}\bigg(g_{(r_1,~r_2)}(\tau)\bigg)=\frac{1}{2}\mathbf{B}_2\big(\langle
r_1\rangle\big)
\end{equation}
where $\langle X\rangle$ is the fractional part of $X\in\mathbb{R}$
with $0\leq\langle X\rangle<1$(\cite{K-L} Chapter 2 Section 1).
\par
Next, we further define
\begin{eqnarray}\label{Delta}
g_2(L)=60\sum_{\omega\in L\setminus\{0\}}\frac{1}{w^4},\quad
g_3(L)=140\sum_{\omega\in L\setminus\{0\}}\frac{1}{w^6},\quad
\Delta(L)=g_2(L)^3-27g_3(L)^2
\end{eqnarray}
and the \textit{elliptic modular function} by
\begin{equation}\label{j}
j(L)=2^63^3\frac{g_2(L)^3}{\Delta(L)}.
\end{equation}

\begin{proposition}\label{gFourier}
\begin{itemize}
\item[(i)] For $\tau\in\mathfrak{H}$ we have the following Fourier expansion formulas
\begin{eqnarray*}
&&g_2(\tau)=g_2([\tau,~1])=(2\pi)^4\frac{1}{2^23}\bigg(1+240\sum_{n=1}^\infty
\sigma_3(n)q_\tau^n\bigg)\\
&&g_3(\tau)=g_3([\tau,~1])=(2\pi)^6\frac{1}{2^33^3}\bigg(1-504\sum_{n=1}^\infty
\sigma_5(n)q_\tau^n\bigg)\\
&&\Delta(\tau)=\Delta([\tau,~1])=(2\pi
i)^{12}q_\tau\prod_{n=1}^\infty(1-q_\tau^n)^{24}
\end{eqnarray*}
where
\begin{equation*}
\sigma_k(n)=\sum_{d>0,~d|n}d^k.
\end{equation*}
\item[(ii)] On $\mathfrak{H}$, $g_2(\tau)$(respectively, $g_3(\tau)$) has zeros only at
$\alpha(\zeta_3)$(respectively, $\alpha(\zeta_4)$) for
$\alpha\in\mathrm{SL}_2(\mathbb{Z})$, and has no poles.
\end{itemize}
\end{proposition}
\begin{proof}
See \cite{Lang} Chapters 3, 4 and 18.
\end{proof}

\begin{remark}\label{jremark}
\begin{itemize}
\item[(i)] By definition (\ref{Fouriereta}) and Proposition \ref{gFourier}(i) we see the relation
\begin{equation}\label{Deltaeta}
\eta(\tau)^{24}=\Delta(\tau).
\end{equation}
\item[(ii)] It follows from definition (\ref{j}) and Proposition
2.1(i) that $j(\tau)=j([\tau,~1])$ has the Fourier expansion with
integer coefficients
\begin{equation*}
j(\tau)=\frac{1}{q_\tau}+744+196884q_\tau+21493760q_\tau^2+864299970q_\tau^3+20245856256q_\tau^4+\cdots.
\end{equation*}
\end{itemize}
\end{remark}

For each positive integer $N$ we let
\begin{equation*}
\Gamma(N)=\bigg\{\begin{pmatrix}a&b\\c&d\end{pmatrix}\in\mathrm{SL}_2(\mathbb{Z})~:~
\begin{pmatrix}a&b\\c&d\end{pmatrix}\equiv\begin{pmatrix}1&0\\0&1\end{pmatrix}\pmod{N}\bigg\}
\end{equation*}
be the \textit{principal congruence subgroup of level $N$}.

\begin{proposition}\label{gmodularity}
We have the following modularity:
\begin{table}[h]
\begin{tabular}{|c||c|c|c|c|c|c|c|c|c|c|}\hline
\backslashbox{\small{Modularity}}{\small{Functions}} & $g_2(\tau)$ &
$g_3(\tau)$ &
 $\eta(\tau)^2$ & $\eta(\tau)^4$ &
$\eta(\tau)^6$ & $\eta(\tau)^{12}$ & $\eta(\tau)^{24}$ &
$j(\tau)$\\\hline\hline with respect to & $\Gamma(1)$ & $\Gamma(1)$
& $\Gamma(12)$ & $\Gamma(6)$ & $\Gamma(3)$ & $\Gamma(2)$ &
$\Gamma(1)$ & $\Gamma(1)$\\\hline weight & $4$ & $6$ & $1$ & $2$ &
$3$ & $6$ & $12$ & $0$\\\hline
\end{tabular}
\end{table}
\end{proposition}
\begin{proof}
See \cite{Lang} Chapter 3 Section 2 and \cite{K-L} Chapter 3 Lemma
5.1.
\end{proof}

For a pair $(r_1,~r_2)\in\mathbb{Q}^2\setminus\mathbb{Z}^2$ we now
define the \textit{Fricke function}
\begin{equation}\label{Fricke}
f_{(r_1,~r_2)}(\tau)=-2^73^5\frac{g_2(\tau)g_3(\tau)\wp(r_1\tau+r_2;~[\tau,~1])}{\Delta(\tau)}\qquad(\tau\in\mathfrak{H}),
\end{equation}
and for $N\geq1$ we let
\begin{equation}\label{defF_N}
\mathcal{F}_N=\mathbb{Q}\bigg(j(\tau),~f_{(r_1,~r_2)}(\tau)~:~(r_1,~r_2)\in\frac{1}{N}\mathbb{Z}^2\setminus\mathbb{Z}^2\bigg)
\end{equation}
which we call the \textit{modular function field of level $N$
rational over $\mathbb{Q}(\zeta_N)$}.

\begin{proposition}\label{meaningF_N}
Let $N\geq1$ and $X(N)$ denote the modular curve
$\Gamma(N)\backslash\mathfrak{H}^*$ where
$\mathfrak{H}^*=\mathfrak{H}\cup\mathbb{P}^1(\mathbb{Q})$. (The
points of $\mathbb{P}^1(\mathbb{Q})=\mathbb{Q}\cup\{\infty\}$ are
called \textit{cusps}.) Then
\begin{itemize}
\item[(i)] $\mathbb{C}\big(X(N)\big)=\mathbb{C}\mathcal{F}_N$.
\item[(ii)] $\mathcal{F}_N$ coincides with the field of
functions in $\mathbb{C}\big(X(N)\big)$ whose Fourier expansions
with respect to $q_\tau^\frac{1}{N}$ have coefficients in
$\mathbb{Q}(\zeta_N)$.
\end{itemize}
\end{proposition}
\begin{proof}
See \cite{Shimura} Propositions 6.1 and 6.9(1).
\end{proof}

\begin{proposition}
$\mathcal{F}_N$ is a Galois extension of
$\mathcal{F}_1=\mathbb{Q}\big(j(\tau)\big)$ whose Galois group is
isomorphic to $\mathrm{GL}_2(\mathbb{Z}/N\mathbb{Z})/\{\pm1_2\}$. In
order to describe the Galois action on $\mathcal{F}_N$ we consider
the decomposition
\begin{equation*}
\mathrm{GL}_2(\mathbb{Z}/N\mathbb{Z})/\{\pm1_2\}=\bigg\{\begin{pmatrix}1&0\\0&d\end{pmatrix}
~:~d\in(\mathbb{Z}/N\mathbb{Z})^*\bigg\}\cdot
\mathrm{SL}_2(\mathbb{Z}/N\mathbb{Z})/\{\pm1_2\}.
\end{equation*}
Here, the matrix
$\left(\begin{smallmatrix}1&0\\0&d\end{smallmatrix}\right)$ acts on
$\sum_{n=-\infty}^\infty c_n q_\tau^\frac{n}{N}\in\mathcal{F}_N$ by
\begin{equation*}
\sum_{n=-\infty}^\infty c_nq_\tau^\frac{n}{N}\mapsto
\sum_{n=-\infty}^\infty c_n^{\sigma_d}q_\tau^\frac{n}{N}
\end{equation*}
where $\sigma_d$ is the automorphism of $\mathbb{Q}(\zeta_N)$
induced by $\zeta_N\mapsto\zeta_N^d$. And, for an element
$\gamma\in\mathrm{SL}_2(\mathbb{Z}/N\mathbb{Z})/\{\pm1_2\}$ let
$\gamma'\in\mathrm{SL}_2(\mathbb{Z})$ be a preimage of $\gamma$ via
the natural surjection
$\mathrm{SL}_2(\mathbb{Z})\rightarrow\mathrm{SL}_2(\mathbb{Z}/N\mathbb{Z})/\{\pm1_2\}$.
Then $\gamma$ acts on $h\in\mathcal{F}_N$ by composition
\begin{equation*}
h\mapsto h\circ\gamma'
\end{equation*}
as linear fractional transformation.
\end{proposition}
\begin{proof}
See \cite{Lang} Chapter 6 Theorem 3.
\end{proof}

\begin{proposition}\label{Siegelmodularity} Let
$N\geq2$. A finite product of Siegel functions
\begin{equation*}
\prod_{r=(r_1,~r_2)\in\frac{1}{N}\mathbb{Z}^2\setminus\mathbb{Z}^2}
g_r(\tau)^{m(r)}
\end{equation*}
belongs to $\mathcal{F}_N$ if
\begin{eqnarray*}
&&\sum_r m(r)(Nr_1)^2\equiv\sum_r
m(r)(Nr_2)^2\equiv0\pmod{\gcd(2,~N)\cdot
N}\\
&&\sum_r m(r)(Nr_1)(Nr_2)\equiv0\pmod{N}\\
&&\sum_r m(r)\cdot \gcd(12,~N)\equiv0\pmod{12}.
\end{eqnarray*}
\end{proposition}
\begin{proof}
See \cite{K-L} Chapter 3 Theorems 5.2 and 5.3.
\end{proof}

\section{Normalization of $\wp'$ by Dedekind
$\eta$-function and some geometry}\label{section3}

Let $L$ be a lattice in $\mathbb{C}$. An elliptic curve $E$ (over
$\mathbb{C}$) with invariant $j(L)$ has an analytic parametrization
in the projective plane $\mathbb{P}^2(\mathbb{C})$ with homogeneous
coordinates $[X:Y:Z]$ via
\begin{eqnarray}\label{analytic}
\varphi~:~\mathbb{C}/L&\stackrel{\sim}{\longrightarrow}&E~:~Y^2Z=4X^3-g_2(L)XZ^2-g_3(L)Z^3\\
z&\mapsto&[\wp(z;~L):\wp'(z;~L):1]\nonumber
\end{eqnarray}
(\cite{Silverman} Chapter VI Proposition 3.6(b)). And we have a
relation
\begin{equation}\label{p'sigma}
\wp'(z;~L)=-\frac{\sigma(2z;~L)}{\sigma(z;~L)^4}
\end{equation}
(\cite{Silverman} p. 166).
\par
Let $N\geq2$ and $L=[\tau,~1]$ with $\tau\in\mathfrak{H}$ as a
variable. Furthermore, let $z=r_1\tau+r_2$ with
$(r_1,~r_2)\in\frac{1}{N}\mathbb{Z}^2\setminus\mathbb{Z}^2$. By
(\ref{analytic}) and (\ref{p'sigma}) the Weierstrass equation
satisfies
\begin{equation}\label{first}
\frac{\sigma(2r_1\tau+2r_2;~[\tau,~1])^2}{\sigma(r_1\tau+r_2;~[\tau,~1])^8}=4\wp(r_1\tau+r_2;~[\tau,~1])^3-g_2(\tau)\wp(r_1\tau+r_2;~[\tau,~1])-g_3(\tau).
\end{equation}
Now we set
\begin{eqnarray}\label{setting}
u(\tau)=\frac{g_2(\tau)^3}{\eta(\tau)^{24}},~
v(\tau)=\frac{g_3(\tau)}{\eta(\tau)^{12}},~
x_{(r_1,~r_2)}(\tau)=-\frac{1}{2^73^5}f_{(r_1,~r_2)}(\tau),~
y_{(r_1,~r_2)}(\tau)=-\frac{g_{(2r_1,~2r_2)}(\tau)}{g_{(r_1,~r_2)}(\tau)^4}.
\end{eqnarray}
Then one can readily check that the equation (\ref{first}) becomes
\begin{equation}\label{second}
u(\tau)v(\tau)^3y_{(r_1,~r_2)}(\tau)^2=4x_{(r_1,~r_2)}(\tau)^3-u(\tau)v(\tau)^2x(\tau)-u(\tau)v(\tau)^4.
\end{equation}
Moreover, by (\ref{Delta}) and (\ref{Deltaeta}) we have an
additional relation
\begin{equation}\label{additional}
u(\tau)-27v(\tau)^2=1.
\end{equation}
Combining (\ref{second}) and (\ref{additional}) we further obtain
some geometric fact. To this end we first need the following lemma.

\begin{lemma}\label{same}
\begin{itemize}
\item[(i)] Let $\tau_1,~\tau_2\in\mathfrak{H}$. Then
$j(\tau_1)=j(\tau_2)$ if and only if $\tau_2=\gamma(\tau_1)$ for
some $\gamma\in\mathrm{SL}_2(\mathbb{Z})$.
\item[(ii)] Let $L$ be a lattice in $\mathbb{C}$ and
$z_1,~z_2\in\mathbb{C}\setminus L$. Then $\wp(z_1;~L)=\wp(z_2;~L)$
if and only if $z_1\equiv\pm z_2\pmod{L}$.
\item[(iii)]
For $(r_1,~r_2)\in\mathbb{Q}^2\setminus\mathbb{Z}^2$ we have
\begin{eqnarray*}
\begin{array}{lllll}
g_{(r_1,~r_2)}(\tau)\circ\begin{pmatrix}0&-1\\1&0\end{pmatrix}&=&
\zeta_{12}^9g_{(r_1,~r_2)\left(\begin{smallmatrix}0&-1\\1&0\end{smallmatrix}\right)}(\tau)
&=&\zeta_{12}^9g_{(r_2,~-r_1)}(\tau)\\
g_{(r_1,~r_2)}(\tau)\circ\begin{pmatrix}1&1\\0&1\end{pmatrix}
&=&\zeta_{12}g_{(r_1,~r_2)\left(\begin{smallmatrix}1&1\\0&1\end{smallmatrix}\right)}(\tau)&=&\zeta_{12}g_{(r_1,~r_1+r_2)}(\tau).
\end{array}
\end{eqnarray*}
\end{itemize}
\end{lemma}
\begin{proof}
See \cite{Cox} Theorem 10.9, Lemma 10.4 and \cite{K-S} Proposition
2.4(2).
\end{proof}

\begin{proposition}\label{surface}
Let $\mathbb{P}^3(\mathbb{C})$ be the projective space with
homogeneous coordinates $[V:X:Y:Z]$ and $S$ be a surface in
$\mathbb{P}^3(\mathbb{C})$ given by the homogeneous equation
\begin{equation*}
(Z^2+27V^2)V^3Y^2=4X^3Z^4-(Z^2+27V^2)V^2XZ^2-(Z^2+27V^2)V^4Z.
\end{equation*}
Let $\Gamma_{1,~4}(N)$ be the congruence subgroup
$\Gamma_1(N)\cap\Gamma(4)$ where $\Gamma_1(N)=\big\{
\left(\begin{smallmatrix}a&b\\c&d\end{smallmatrix}\right)
\in\mathrm{SL}_2(\mathbb{Z})~:~\left(\begin{smallmatrix}a&b\\c&d\end{smallmatrix}\right)\equiv\left(\begin{smallmatrix}1&*\\0&1\end{smallmatrix}\right)
\pmod{N}\big\}$ and $X_{1,~4}(N)$ be its corresponding modular curve
$\Gamma_{1,~4}(N)\backslash\mathfrak{H}^*$. If $4~|~N$, then we have
a holomorphic map
\begin{eqnarray*}
\iota~:~X_{1,~4}(N)&\longrightarrow& S\\
\tau&\mapsto&[v(\tau):x_{(0,~\frac{1}{N})}(\tau):y_{(0,~\frac{1}{N})}(\tau):1].\nonumber
\end{eqnarray*}
In particular, if $M$ is the image of
$\{\textrm{cusps},~\alpha(\zeta_3),~\alpha(\zeta_4):\alpha\in\mathrm{SL}_2(\mathbb{Z})\}$
via the natural quotient map $\mathfrak{H}^*\rightarrow
X_{1,~4}(N)$, then the restriction morphism
$\iota:X_{1,~4}(N)\setminus M\rightarrow S$ gives an embedding into
$\mathbb{P}^3(\mathbb{C})$.
\end{proposition}
\begin{proof}
Let $4~|~N$. Since the functions
$v(\tau),~x_{(0,~\frac{1}{N})}(\tau),~y_{(0,~\frac{1}{N})}(\tau),~1$
are not all identically zero, the map $\iota$ extends to a
holomorphic map defined on all of the modular curve
$X_{1,~4}(N)$(\cite{Miranda} Chapter V Lemma 4.2) and its image is
contained in $S$ by (\ref{second}) and (\ref{additional}) provided
that it is well-defined.
\par
Since $v(\tau)\in\mathbb{C}\big(X(2)\big)$ by Proposition
\ref{gmodularity}, $v(\tau)\in\mathbb{C}\big(X_{1,~4}(N)\big)$. And,
$x_{(0,~\frac{1}{N})}(\tau)\in\mathcal{F}_N$ by definition
(\ref{defF_N}) and $y_{(0,~\frac{1}{N})}(\tau)\in\mathcal{F}_N$ by
Proposition \ref{Siegelmodularity}. Here we observe that
$\Gamma_1(N)=\big\langle\Gamma(N),~T=\left(\begin{smallmatrix}1&1\\0&1\end{smallmatrix}\right)\big\rangle$.
We then obtain by definition (\ref{Fricke}) and Proposition
\ref{gmodularity} that
\begin{eqnarray*}
x_{(0,~\frac{1}{N})}(\tau)\circ T&=&
\frac{g_2\big(T(\tau)\big)g_3\big(T(\tau)\big)\wp\big(\frac{1}{N};~[T(\tau),~1]\big)}{\Delta\big(T(\tau)\big)}\\
&=&\frac{g_2(\tau)g_3(\tau)\wp(\frac{1}{N};~[\tau+1,~1])}{\Delta(\tau)}\\
&=&\frac{g_2(\tau)g_3(\tau)\wp(\frac{1}{N};~[\tau,~1])}{\Delta(\tau)}~=~x_{(0,~\frac{1}{N})}(\tau),
\end{eqnarray*}
from which we get
$x_{(0,~\frac{1}{N})}(\tau)\in\mathbb{C}\big(X_1(N)\big)\subseteq\mathbb{C}\big(X_{1,~4}(N)\big)$.
On the other hand, if
$\gamma\in\Gamma_{1,~4}(N)(\subseteq\Gamma_1(N))$, then $\gamma$ is
of the form $(\gamma_1T^{e_1})\cdots(\gamma_nT^{e_n})$ for some
$\gamma_1,\cdots,\gamma_n\in\Gamma(N)$ and
$e_1,\cdots,e_n\in\mathbb{Z}$ such that
$e_1+\cdots+e_n\equiv0\pmod4$. Thus we derive from the fact
$y_{(0,~\frac{1}{N})}(\tau)\in\mathcal{F}_N$ that
\begin{eqnarray*}
y_{(0,~\frac{1}{N})}(\tau)\circ\gamma&=&
\bigg(\frac{g_{(0,~\frac{2}{N})}(\tau)}{g_{(0,~\frac{1}{N})}(\tau)^4}\bigg)\circ\gamma
=\bigg(\frac{g_{(0,~\frac{2}{N})}(\tau)}{g_{(0,~\frac{1}{N})}(\tau)^4}\bigg)\circ(\gamma_1T^{e_1})\cdots(\gamma_nT^{e_n})\\
&=&\zeta_{12}^{-3(e_1+\cdots+e_n)}\frac{g_{(0,~\frac{2}{N})}(\tau)}{g_{(0,~\frac{1}{N})}(\tau)^4}\quad\textrm{by
Lemma \ref{same}}\\
&=&\frac{g_{(0,~\frac{2}{N})}(\tau)}{g_{(0,~\frac{1}{N})}(\tau)^4}~=~y_{(0,~\frac{1}{N})}(\tau)
\quad\textrm{by the fact $e_1+\cdots+e_n\equiv0\pmod4$},
\end{eqnarray*}
which yields
$y_{(0,~\frac{1}{N})}(\tau)\in\mathbb{C}\big(X_{1,~4}(N)\big)$.
Hence the map $\iota$ is well-defined.
\par
Now, assume $\iota(\tau_1)=\iota(\tau_2)$ for some points
$\tau_1,~\tau_2\in\mathfrak{H}^*\setminus\{\textrm{cusps},~\alpha(\zeta_3),~\alpha(\zeta_4):\alpha\in\mathrm{SL}_2(\mathbb{Z})\}$.
Then we deduce by definitions (\ref{j}) and (\ref{setting}) that
\begin{equation*}
j(\tau_1)=j(\tau_2),\quad
f_{(0,~\frac{1}{N})}(\tau_1)=f_{(0,~\frac{1}{N})}(\tau_2)\quad\textrm{and}\quad
\frac{g_{(0,~\frac{2}{N})}(\tau_1)}{g_{(0,~\frac{1}{N})}(\tau_1)^4}=
\frac{g_{(0,~\frac{2}{N})}(\tau_2)}{g_{(0,~\frac{1}{N})}(\tau_2)^4}.
\end{equation*}
(Note that all functions $v(\tau)$, $x_{(0,~\frac{1}{N})}(\tau)$,
$y_{(0,~\frac{1}{N})}(\tau)$ do not have poles on $\mathfrak{H}$.)
So we get $\tau_2=\gamma(\tau_1)$ for some
$\gamma=\left(\begin{smallmatrix}a&b\\c&d\end{smallmatrix}\right)\in\mathrm{SL}_2(\mathbb{Z})$
by the fact $j(\tau_1)=j(\tau_2)$ and Lemma \ref{same}(i). Moreover,
it follows from the fact
$f_{(0,~\frac{1}{N})}(\tau_1)=f_{(0,~\frac{1}{N})}(\tau_2)$ and
definition (\ref{Fricke}) that
\begin{eqnarray*}
&&\frac{g_2(\tau_1)g_3(\tau_1)\wp(\frac{1}{N};~[\tau_1,~1])}{\Delta(\tau_1)}
=\frac{g_2(\tau_2)g_3(\tau_2)\wp(\frac{1}{N};~[\tau_2,~1])}{\Delta(\tau_2)}\\
&=&\frac{g_2\big(\gamma(\tau_1)\big)g_3\big(\gamma(\tau_1)\big)\wp(\frac{1}{N};~[\gamma(\tau_1),~1])}{\Delta\big(\gamma(\tau_1)\big)}
=\frac{g_2(\tau_1)g_3(\tau_1)\wp(\frac{1}{N}(c\tau_1+d);~[\tau_1,~1])}{\Delta(\tau_1)}
\end{eqnarray*}
due to Proposition \ref{gmodularity} and definition
(\ref{p-function}). And, we achieve by Proposition
\ref{gFourier}(ii) and Lemma \ref{same}(ii)
\begin{equation*}
\frac{1}{N}\equiv\pm\frac{1}{N}(c\tau_1+d)\pmod{[\tau_1,~1]},
\end{equation*}
from which we have $c\equiv0\pmod{N}$ and $d\equiv\pm1\pmod{N}$.
Hence the relation $\det(\gamma)=ad-bc=1$ implies $a\equiv
d\equiv\pm1\pmod{N}$. Thus we may assume that $\gamma$ belongs to
the congruence subgroup $\Gamma_1(N)$ because $\gamma$ and $-\gamma$
give rise to the same linear fractional transformation. On the other
hand, since
$\Gamma_1(N)=\big\langle\Gamma(N),~T=\left(\begin{smallmatrix}1&1\\0&1\end{smallmatrix}\right)\big\rangle$,
$\gamma$ is of the form $(\gamma_1T^{e_1})\cdots(\gamma_nT^{e_n})$
for some $\gamma_1,\cdots,\gamma_n\in\Gamma(N)$ and
$e_1,\cdots,e_n\in\mathbb{Z}$ such that $e_1+\cdots+e_n\equiv
b\pmod{N}$. Furthermore, from the fact that
$\frac{g_{(0,~\frac{2}{N})}(\tau_1)}{g_{(0,~\frac{1}{N})}(\tau_1)^4}=
\frac{g_{(0,~\frac{2}{N})}(\tau_2)}{g_{(0,~\frac{1}{N})}(\tau_2)^4}$
and
$\frac{g_{(0,~\frac{2}{N})}(\tau)}{g_{(0,~\frac{1}{N})}(\tau)^4}\in\mathcal{F}_N$
we derive
\begin{eqnarray*}
&&\frac{g_{(0,~\frac{2}{N})}(\tau_1)}{g_{(0,~\frac{1}{N})}(\tau_1)^4}=\frac{g_{(0,~\frac{2}{N})}(\tau_2)}{g_{(0,~\frac{1}{N})}(\tau_2)^4}
=\frac{g_{(0,~\frac{2}{N})}\big(\gamma(\tau_1)\big)}{g_{(0,~\frac{1}{N})}\big(\gamma(\tau_1)\big)^4}
=\bigg(\frac{g_{(0,~\frac{2}{N})}(\tau)}{g_{(0,~\frac{1}{N})}(\tau)^4}\bigg)\circ\gamma(\tau_1)\\
&=&\bigg(\frac{g_{(0,~\frac{2}{N})}(\tau)}{g_{(0,~\frac{1}{N})}(\tau)^4}\bigg)\circ(\gamma_1T^{e_1})\cdots(\gamma_nT^{e_n})(\tau_1)
=\zeta_{12}^{-3(e_1+\cdots+e_n)}\frac{g_{(0,~\frac{2}{N})}(\tau_1)}{g_{(0,~\frac{1}{N})}(\tau_1)^4}
\quad\textrm{by Lemma \ref{same}(iii).}
\end{eqnarray*}
Therefore $e_1+\cdots+e_n\equiv0\pmod{4}$, and so $b\equiv0\pmod{4}$
because $e_1+\cdots+e_n\equiv b\pmod{N}$ and $4~|~N$. We then see
that $\gamma$ belongs to the congruence subgroup $\Gamma_{1,~4}(N)$,
which implies that $\tau_1$ and $\tau_2$ represent the same point on
$X_{1,~4}(N)\setminus M$. This proves that the restriction morphism
is indeed an embedding as desired.
\end{proof}

\begin{remark}
\begin{itemize}
\item[(i)] Unfortunately, however, the morphism $\iota:X_{1,~4}(N)\rightarrow S$ is
not injective. For instance, one can check it with the cusps.
Indeed, let $s$ be a cusp of width $w$ and
$\alpha=\left(\begin{smallmatrix}a&b\\c&d\end{smallmatrix}\right)$
be an element of $\mathrm{SL}_2(\mathbb{Z})$ such that
$\alpha(\infty)=s$. Then we get that
\begin{eqnarray}
\mathrm{ord}_s\big(v(\tau)\big)&=&w\times\mathrm{ord}_{q_\tau}\big(v(\tau)\circ\alpha\big)
=w\times\tfrac{1}{2}\mathrm{ord}_{q_\tau}\big((u(\tau)-1)\circ\alpha\big)\quad\textrm{by
the relation
(\ref{additional})}\nonumber\\
&=&w\times\tfrac{1}{2}\mathrm{ord}_{q_\tau}\big((\tfrac{1}{2^63^3}j(\tau)-1)\circ\alpha\big)
\quad\textrm{by definitions (\ref{setting}) and (\ref{j})}\nonumber\\
&=&w\times\tfrac{1}{2}\mathrm{ord}_{q_\tau}\big(\tfrac{1}{2^63^3}j(\tau)-1\big)\quad\textrm{by
Proposition \ref{gmodularity}}\nonumber\\
&=&w\times\big(-\tfrac{1}{2}\big)\quad\textrm{by Remark
\ref{jremark}(ii).}\label{vorder}
\end{eqnarray}
And, we further obtain that
\begin{eqnarray}
&&\mathrm{ord}_s\bigg(y_{(0,~\frac{1}{N})}(\tau)\bigg)=w\times\mathrm{ord}_{q_\tau}\bigg(
\frac{g_{(0,~\frac{2}{N})}(\tau)}{g_{(0,~\frac{1}{N})}(\tau)^4}\circ\alpha
\bigg)\nonumber\\
&=&w\times\mathrm{ord}_{q_\tau}\bigg(
\frac{g_{(\frac{2c}{N},~\frac{2d}{N})}(\tau)}{g_{(\frac{c}{N},~\frac{d}{N})}(\tau)^4}\bigg)\quad\textrm{by
Lemma \ref{same} because }\mathrm{SL}_2(\mathbb{Z})=\big\langle
\left(\begin{smallmatrix}0&-1\\1&0\end{smallmatrix}\right),~
\left(\begin{smallmatrix}1&1\\0&1\end{smallmatrix}\right)\big\rangle\nonumber\\
&=&w\times\bigg(\tfrac{1}{2}\mathbf{B}_2\big(\langle\tfrac{2c}{N}\rangle\big)-
4\cdot\tfrac{1}{2}\mathbf{B}_2\big(\langle\tfrac{c}{N}\rangle\big)\bigg)
\quad\textrm{by the formula (\ref{order})}\nonumber\\
&=&\left\{\begin{array}{ll}
w\times\big(\langle\tfrac{c}{N}\rangle-\frac{1}{4}\big)&\textrm{if
$0\leq\langle\tfrac{c}{N}\rangle<\tfrac{1}{2}$ (, so
$\langle\tfrac{2c}{N}\rangle=2\langle\tfrac{c}{N}\rangle$)}\vspace{0.2cm}\\
w\times\big(-\langle\tfrac{c}{N}\rangle+\frac{3}{4}\big)&\textrm{if
$\tfrac{1}{2}\leq\langle\tfrac{c}{N}\rangle<1$ (, so
$\langle\tfrac{2c}{N}\rangle=2\langle\tfrac{c}{N}\rangle-1$).}
\end{array}\right.\nonumber
\end{eqnarray}
It then follows
\begin{equation}\label{yorder}
w\times\big(-\tfrac{1}{4}\big)\leq\mathrm{ord}_s\bigg(y_{(0,~\frac{1}{N})}(\tau)\bigg)
\leq w\times\tfrac{1}{4},
\end{equation}
whose first equality holds if and only if
$\langle\frac{c}{N}\rangle=0$. On the other hand, we have by the
morphism in Proposition \ref{surface}
\begin{eqnarray}\label{eqn}
&&\big(1+27v(\tau)^2\big)v(\tau)^3y_{(0,~\frac{1}{N})}(\tau)^2\nonumber\\
&=&4x_{(0,~\frac{1}{N})}(\tau)^3-\big(1+27v(\tau)^2\big)v(\tau)^2x_{(0,~\frac{1}{N})}(\tau)-\big(1+27v(\tau)^2\big)v(\tau)^4.
\end{eqnarray}
Let $t=\mathrm{ord}_s\big(x_{(0,~\frac{1}{N})}(\tau)\big)$ and
assume $\langle\frac{c}{N}\rangle\neq0$. Observe that there exist at
least two such inequivalent cusps with respect to
$\Gamma_{1,~4}(N)$, for example $s=1,~-1$. Then we derive by
(\ref{vorder}) and (\ref{yorder})
\begin{equation}\label{order3}
w\times(-3)<\mathrm{ord}_s\big(\textrm{LHS of (\ref{eqn})}\big).
\end{equation}
In this case, if $t\neq w\times(-1)$, then one can readily check
that
\begin{equation*}
\mathrm{ord}_s\big(\textrm{RHS of (\ref{eqn})}\big)=
\left\{\begin{array}{ll} w\times(-3)&\textrm{if $t>w\times(-1)$}\\
3t&\textrm{if $t<w\times(-1)$,}
\end{array}\right.
\end{equation*}
because $\mathrm{ord}_s(\cdot)$ is a valuation on the function field
$\mathbb{C}\big(X_{1,~4}(N)\big)$. Hence, this fact and
(\ref{order3}) lead to a contradiction to the identity (\ref{eqn}),
and so $t=w\times(-1)$. Therefore we claim that
\begin{eqnarray*}
\iota(s)&=&\bigg[\bigg(\frac{v(\tau)\circ\alpha}{x_{(0,~\frac{1}{N})}(\tau)\circ\alpha}\bigg)\bigg|_{q_\tau=0}
:1:\bigg(\frac{y_{(0,~\frac{1}{N})}(\tau)\circ\alpha}{x_{(0,~\frac{1}{N})}(\tau)\circ\alpha}\bigg)\bigg|_{q_\tau=0}:
\bigg(\frac{1}{x_{(0,~\frac{1}{N})}(\tau)\circ\alpha}\bigg)\bigg|_{q_\tau=0}\bigg]\\
&=&[0:1:0:0]
\end{eqnarray*}
(\cite{Miranda} Chapter V Lemma 4.2), from which we conclude that
the morphism is not injective.
\item[(ii)] As for the possible zeros of
$x_{(0,~\frac{1}{N})}(\tau)$ in $\mathfrak{H}$, it is probable that
the restriction morphism
$\iota:\Gamma_{1,~4}(N)\backslash\mathfrak{H}\rightarrow S$ is
injective. For example, if $N=4$, then the image of
$\{\alpha(\zeta_3),~\alpha(\zeta_4):\alpha\in\mathrm{SL}_2(\mathbb{Z})\}$
via the natural quotient map $\mathfrak{H}^*\rightarrow X_{1,~4}(N)$
consists of $20$ points, namely
\begin{eqnarray*}
&&\big\{\zeta_3,~\zeta_3+1,~\zeta_3+2,~\zeta_3+3,~
\tfrac{1}{-\zeta_3+1},~
\tfrac{1}{-\zeta_3+2},~\tfrac{2\zeta_3-1}{3\zeta_3+2},~
\tfrac{\zeta_3-2}{\zeta_3-1},\\
&&\phantom{\{}\zeta_4,~\zeta_4+1,~\zeta_4+2,~\zeta_4+3,~
\tfrac{1}{-\zeta_4+1},~ \tfrac{1}{-\zeta_4+2},~
\tfrac{1}{-\zeta_4+3},~ \tfrac{\zeta_4+1}{\zeta_4+2},~
\tfrac{\zeta_4-1}{-\zeta_4+2},~ \tfrac{\zeta_4-2}{\zeta_4-1},~
\tfrac{\zeta_4+2}{-\zeta_4-1},~ \tfrac{2\zeta_4+1}{3\zeta_4+2}
\big\}.
\end{eqnarray*}
And, by numerical computation one can show that the value of
$y_{(0,~\frac{1}{N})}(\tau)$ at each point is distinct, which
implies that the restriction morphism $\iota$ is injective.
\item[(iii)]
It seems that in the above proposition there should be an additional
hidden relation between $v(\tau)$ and $y_{(0,~\frac{1}{N})}(\tau)$
because $y_{(0,~\frac{1}{N})}(\tau)$ is a modular unit(see
\cite{K-S} or \cite{K-L}). That is, $y_{(0,~\frac{1}{N})}(\tau)$
satisfies a monic polynomial
\begin{equation*}
f(Y)=\prod_{\gamma\in\mathrm{Gal}(\mathcal{F}_N/\mathcal{F}_1)}\big(Y-y_{(0,~\frac{1}{N})}(\tau)^\gamma\big)
\end{equation*}
with coefficients in $\mathbb{Q}[v(\tau)]$. If we consider $f(Y)$ as
a polynomial $f(V,~Y)$ of $Y$ and $V$, then the intersection of $S$
and a hypersurface obtained from $f(V,~Y)$ may be a (singular) curve
in $\mathbb{P}^3(\mathbb{C})$.
\end{itemize}
\end{remark}

\section{Explicit description of Shimura's reciprocity law}

We shall briefly review an algorithm of determining all conjugates
of the singular value of a modular function, from which we can find
conjugates of singular values of certain Siegel functions due to
\cite{Gee}(or \cite{Stevenhagen}) and \cite{J-K-S}.

Throughout this section by $K$ we mean an imaginary quadratic field
with discriminant $d_K$ and define
\begin{eqnarray}\label{theta}
\theta=\left\{\begin{array}{ll}\frac{\sqrt{d_K}}{2}&\textrm{for}~d_K\equiv0\pmod{4}\\
\frac{-1+\sqrt{d_K}}{2}&\textrm{for}~
d_K\equiv1\pmod{4},\end{array}\right.
\end{eqnarray}
from which we get $\mathcal{O}_K=[\theta,~1]$. And, we denote by $H$
and $K_{(N)}$ the Hilbert class field and the ray class field modulo
$N\mathcal{O}_K$ over $K$ for an integer $N\geq1$, respectively.

\begin{proposition}\label{cm}
By the main theorem of complex multiplication we derive that
\begin{itemize}
\item[(i)] $H=K\big(j(\theta)\big)$.
\item[(ii)] $K_{(N)}=K\big(h(\theta)~:~h\in\mathcal{F}_N~\textrm{is defined and
finite at $\theta$}\big)$.
\item[(iii)] If $d_K\leq-7$ and $N\geq2$, then
$K_{(N)}=H\big(f_{(0,~\frac{1}{N})}(\theta)\big)$.
\end{itemize}
\end{proposition}
\begin{proof}
See \cite{Lang} Chapter 10.
\end{proof}

\begin{proposition}\label{rayGalois}
Let $\min(\theta,~\mathbb{Q})=X^2+B_\theta
X+C_\theta\in\mathbb{Z}[X]$. For every positive integer $N$ the
matrix group
\begin{equation*}W_{N,~\theta}=\bigg\{\begin{pmatrix}t-B_\theta s &
-C_\theta
s\\s&t\end{pmatrix}\in\mathrm{GL}_2(\mathbb{Z}/N\mathbb{Z})~:~t,~s\in\mathbb{Z}/N\mathbb{Z}\bigg\}
\end{equation*}
gives rise to the surjection
\begin{eqnarray*}
W_{N,~\theta}&\longrightarrow&\mathrm{Gal}(K_{(N)}/H)\\
\alpha&\mapsto&\bigg(h(\theta)\mapsto h^\alpha(\theta)\bigg)
\end{eqnarray*}
where $h\in\mathcal{F}_N$ is defined and finite at $\theta$. If
$d_K\leq-7$, then the kernel is $\{\pm1_2\}$.
\end{proposition}
\begin{proof}
See \cite{Gee} or \cite{Stevenhagen}.
\end{proof}

Under the properly equivalent relation, primitive positive definite
quadratic forms $aX^2+bXY+cY^2\in\mathbb{Z}[X,~Y]$ of discriminant
$d_K$ determine a group $\mathrm{C}(d_K)$, called the \textit{form
class group of discriminant $d_K$}. We identify $\mathrm{C}(d_K)$
with the set of all \textit{reduced quadratic forms}, which are
characterized by the conditions
\begin{equation}\label{reduced}
-a<b\leq a<c\quad\textrm{or}\quad 0\leq b\leq a=c
\end{equation}
together with the discriminant relation
\begin{equation}\label{disc}
b^2-4ac=d_K.
\end{equation}
From the above two conditions for reduced quadratic forms we deduce
\begin{equation}\label{bound a}
1\leq a\leq\sqrt{\tfrac{-d_K}{3}}.
\end{equation}
As is well-known(\cite{Cox}) $\mathrm{C}(d_K)$ is isomorphic to
$\textrm{Gal}(H/K)$. Now, for a reduced quadratic form
$Q=aX^2+bXY+cY^2$ of discriminant $d_K$ we define a CM-point
\begin{equation}\label{theta_Q}
\theta_Q=\frac{-b+\sqrt{d_K}}{2a}.
\end{equation}
Furthermore, we define
$\beta_Q=(\beta_p)_p\in\prod_{p~:~\textrm{prime}}\textrm{GL}_2(\mathbb{Z}_p)$
as
\begin{eqnarray}\label{u1}
\beta_p=\left\{\begin{array}{ll}
\begin{pmatrix}a&\frac{b}{2}\\0&1\end{pmatrix}&\textrm{if $p\nmid a$}\\
\begin{pmatrix}-\frac{b}{2}&-c\\1&0\end{pmatrix}&\textrm{if $p\mid a$ and $p\nmid c$}\\
\begin{pmatrix}-\frac{b}{2}-a&-\frac{b}{2}-c\\1&-1\end{pmatrix}&\textrm{if $p\mid a$ and $p\mid
c$}\end{array}\right.\quad\textrm{for}~d_K\equiv0\pmod{4}
\end{eqnarray}
and
\begin{eqnarray}\label{u2}
\beta_p=\left\{\begin{array}{ll}
\begin{pmatrix}a&\frac{b-1}{2}\\0&1\end{pmatrix}&\textrm{if}~p\nmid a\\
\begin{pmatrix}\frac{-b-1}{2}&-c\\1&0\end{pmatrix}&\textrm{if $p\mid a$ and $p\nmid c$}\\
\begin{pmatrix}\frac{-b-1}{2}-a&\frac{1-b}{2}-c\\1&-1\end{pmatrix}&\textrm{if $p\mid a$ and $p\mid
c$}\end{array}\right.\quad\textrm{for}~d_K\equiv1\pmod{4}.
\end{eqnarray}

\begin{proposition}\label{conjugate}
Assume $d_K\leq-7$ and $N\geq1$. Then we have a bijective map
\begin{eqnarray*}
\begin{array}{cccc}
W_{N,~\theta}/\{\pm1_2\}\times\mathrm{C}(d_K)&\longrightarrow&
\mathrm{Gal}(K_{(N)}/K)&\\
(\alpha,~ Q)&\longmapsto&\bigg(h(\theta)\mapsto
h^{\alpha\beta_Q}(\theta_Q)\bigg).\end{array}
\end{eqnarray*}
Here, $h\in\mathcal{F}_N$ is defined and finite at $\theta$. The
action of $\alpha$ on $\mathcal{F}_N$ is the action as an element of
$\mathrm{GL}_2(\mathbb{Z}/N\mathbb{Z})/\{\pm1_2\}\cong\mathrm{Gal}(\mathcal{F}_N/\mathcal{F}_1)$.
And, as for $\beta_Q$ we note that there exists
$\beta\in\mathrm{GL}_2^+(\mathbb{Q})\cap \mathrm{M}_2(\mathbb{Z})$
such that $\beta\equiv \beta_p\pmod{N\mathbb{Z}_p}$ for all primes
$p$ dividing $N$ by the Chinese remainder theorem. Thus the action
of $\beta_Q$ on $\mathcal{F}_N$ is understood as that of $\beta$
which is also an element of
$\mathrm{GL}_2(\mathbb{Z}/N\mathbb{Z})/\{\pm1_2\}$.
\end{proposition}
\begin{proof}
See \cite{J-K-S} Theorem 3.4.
\end{proof}

We need some transformation formulas of Siegel functions to apply
the above proposition.

\begin{proposition}\label{F_N}
Let $N\geq 2$. For $(r_1,~r_2)\in\frac{1}{N}\mathbb{Z}^2\setminus
\mathbb{Z}^2$ the function $g_{(r_1,~r_2)}(\tau)^{12N}$ satisfies
\begin{equation*}
g_{(r_1,~r_2)}(\tau)^{12N}=g_{(-r_1,~-r_2)}(\tau)^{12N}=g_{(\langle
r_1\rangle,~\langle r_2\rangle)}(\tau)^{12N}.
\end{equation*}
And, it belongs to $\mathcal{F}_N$ and $\alpha$ in
$\mathrm{GL}_2(\mathbb{Z}/N\mathbb{Z})/\{\pm1_2\}\cong\mathrm{Gal}(\mathcal{F}_N/
\mathcal{F}_1)$ acts on the function by
\begin{equation*}
\bigg(g_{(r_1,~r_2)}(\tau)^{12N}\bigg)^\alpha=
g_{(r_1,~r_2)\alpha}(\tau)^{12N}.
\end{equation*}
\end{proposition}
\begin{proof}
See \cite{K-S} Proposition 2.4 and Theorem 2.5.
\end{proof}

\section{Generation of ray class fields by torsion points of elliptic curves}

Let $K$ be an imaginary quadratic field with discriminant $d_K$ and
$\theta$ as in (\ref{theta}). Here, we shall construct the ray class
field $K_{(N)}$ by adjoining to $K$ some $N$-torsion point of
certain elliptic curve, if $d_K\leq-39$, $N\geq8$ and $4~|~N$.
\par
For convenience we set
\begin{equation*}
D=\sqrt{\tfrac{-d_K}{3}}\quad\textrm{and}\quad A=|e^{2\pi
i\theta}|=e^{-\pi\sqrt{-d_K}}.
\end{equation*}

\begin{lemma}\label{ineq}
We have the following inequalities:
\begin{itemize}
\item[(i)] If $d_K\leq-7$, then
\begin{equation}\label{A}
\frac{1}{1-A^{\frac{X}{a}}}<1+A^\frac{X}{1.03a}
\end{equation}
for $1\leq a\leq D$ and all $X\geq\frac{1}{2}$.
\item[(ii)] $1+X<e^X$ for all $X>0$.
\end{itemize}
\end{lemma}
\begin{proof}
\begin{itemize}
\item[(i)] The inequality (\ref{A}) is equivalent to
\begin{equation*}
A^{\frac{X}{a}\frac{3}{103}}+A^{\frac{X}{a}}<1.
\end{equation*}
Since $A=e^{-\pi\sqrt{-d_K}}\leq e^{-\pi\sqrt{7}}<1$,  $1\leq a\leq
D$ and $X\geq\frac{1}{2}$, we obtain that
\begin{equation*}
A^{\frac{X}{a}\frac{3}{103}}+A^{\frac{X}{a}}\leq
A^{\frac{1}{2D}\frac{3}{103}}+A^{\frac{1}{2D}}=
e^{-\frac{\pi\sqrt{3}}{2}\frac{3}{103}}+e^{\frac{-\pi\sqrt{3}}{2}}<1
\end{equation*}
by the fact $A^\frac{1}{D}=e^{-\pi\sqrt{3}}$. This proves (i).
\item[(ii)] Immediate.
\end{itemize}
\end{proof}

\begin{lemma}\label{newlemma}
Assume that $d_K\leq-39$ and $N\geq8$. Let $Q=aX^2+bXY+cY^2$ be a
reduced quadratic form of discriminant $d_K$. If $a\geq2$, then the
inequality
\begin{equation*}
\bigg|\frac{g_{(\frac{2s}{N},~\frac{2t}{N})}(\theta_Q)}{g_{(\frac{s}{N},~\frac{t}{N})}(\theta_Q)^4}\bigg|
<\bigg|\frac{g_{(0,~\frac{2}{N})}(\theta)}{g_{(0,~\frac{1}{N})}(\theta)^4}\bigg|.
\end{equation*}
holds for $(s,~t)\in\mathbb{Z}^2$ with $(2s,~2t)\not\in
N\mathbb{Z}^2$.
\end{lemma}
\begin{proof}
We may assume $0\leq s\leq\tfrac{N}{2}$ by Proposition \ref{F_N}.
And, observe that $2\leq a\leq D$ by (\ref{bound a}) and $A\leq
e^{-\pi\sqrt{39}}<1$. From the Fourier expansion formula
(\ref{FourierSiegel}) we establish that
\begin{eqnarray*}
&&\bigg|\bigg(\frac{g_{(0,~\frac{2}{N})}(\theta)}{g_{(0,~\frac{1}{N})}(\theta)^4}\bigg)^{-1}
\bigg(\frac{g_{(\frac{2s}{N},~\frac{2t}{N})}(\theta_Q)}{g_{(\frac{s}{N},~\frac{t}{N})}(\theta_Q)^4}\bigg)\bigg|\\
&\leq& A^{(\frac{1}{4}-\frac{1}{4a}+\frac{s}{aN})}
T(N,~s,~t)\prod_{n=1}^\infty
\frac{(1+A^n)^8(1+A^{\frac{1}{a}(n+\frac{2s}{N})})(1+A^{\frac{1}{a}(n-\frac{2s}{N})})}
{(1-A^n)^2(1-A^{\frac{1}{a}(n+\frac{s}{N})})^4(1-A^{\frac{1}{a}(n-\frac{s}{N})})^4}\\
&\leq& A^{(\frac{1}{4}-\frac{1}{4a})} T(N,~s,~t)\prod_{n=1}^\infty
\frac{(1+A^n)^8(1+A^{\frac{n}{a}})(1+A^{\frac{1}{a}(n-1)})}
{(1-A^n)^2(1-A^{\frac{n}{a}})^4(1-A^{\frac{1}{a}(n-\frac{1}{2})})^4}\quad\textrm{by the fact $0\leq s\leq\frac{N}{2}$}\\
&\leq&A^\frac{1}{8}T(N,~s,~t) \prod_{n=1}^\infty
\frac{(1+A^n)^8(1+A^{\frac{n}{D}})(1+A^{\frac{1}{D}(n-1)})}
{(1-A^n)^2(1-A^{\frac{n}{D}})^4(1-A^{\frac{1}{D}(n-\frac{1}{2})})^4}\quad\textrm{by
the fact $2\leq a\leq D$}
\end{eqnarray*}
where
\begin{equation*}
T(N,~s,~t)=\bigg|\frac{(1-\zeta_N)^4}{1-\zeta_N^2}\bigg|\bigg|\frac{1-e^{2\pi
i(\frac{2s}{N}\theta_Q+\frac{2t}{N})}}{\big(1-e^{2\pi
i(\frac{s}{N}\theta_Q+\frac{t}{N})}\big)^4} \bigg|=
\bigg|\frac{(1-\zeta_N)^3}{1+\zeta_N}\bigg|\bigg|\frac{1+e^{2\pi
i(\frac{s}{N}\theta_Q+\frac{t}{N})}}{\big(1-e^{2\pi
i(\frac{s}{N}\theta_Q+\frac{t}{N})}\big)^3} \bigg|.
\end{equation*}
If $s=0$, then
\begin{equation*}
T(N,~s,~t)=\bigg|\bigg(\frac{1-\zeta_N}{1-\zeta_N^t}\bigg)^3\bigg|\bigg|\frac{1+\zeta_N^t}{1+\zeta_N}\bigg|=
\bigg|\bigg(\frac{\sin\frac{\pi}{N}}{\sin\frac{t\pi}{N}}\bigg)^3\bigg|\bigg|\frac{\cos\frac{t\pi}{N}}{\cos\frac{\pi}{N}}\bigg|\leq1.
\end{equation*}
If $s\neq0$, then
\begin{eqnarray*}
T(N,~s,~t)&\leq&\bigg|\frac{(1-\zeta_N)^3}{1+\zeta_N}\bigg|\frac{1+A^{\frac{1}{Na}}}{(1-A^{\frac{1}{Na}})^3}\quad\textrm{by the fact $1\leq s\leq\frac{N}{2}$}\\
&\leq&\bigg|\frac{(1-\zeta_N)^3}{1+\zeta_N}\bigg|\frac{1+A^{\frac{1}{ND}}}{(1-A^{\frac{1}{ND}})^3}
\quad\textrm{by the fact $2\leq a\leq D$}\\
&=&\frac{4\sin^3\frac{\pi}{N}}{\cos\frac{\pi}{N}}\frac{1+e^{-\frac{\pi\sqrt{3}}{N}}}{(1-e^{-\frac{\pi\sqrt{3}}{N}})^3}
\quad\textrm{by the fact $A^\frac{1}{D}=e^{-\pi\sqrt{3}}$}\\
&<&3.05\quad\textrm{from the graph on $N\geq8$}.
\end{eqnarray*}
 Therefore we
achieve that
\begin{eqnarray*}
&&\bigg|\bigg(\frac{g_{(0,~\frac{2}{N})}(\theta)}{g_{(0,~\frac{1}{N})}(\theta)^4}\bigg)^{-1}
\bigg(\frac{g_{(\frac{2s}{N},~\frac{2t}{N})}(\theta_Q)}{g_{(\frac{s}{N},~\frac{t}{N})}(\theta_Q)^4}\bigg)\bigg|\\
&<&3.05A^\frac{1}{8}\prod_{n=1}^\infty\frac{
(1+A^n)^8(1+A^\frac{n}{D})(1+A^{\frac{1}{D}(n-1)})}{(1+A^\frac{n}{1.03})^{-2}(1+A^\frac{n}{1.03D})^{-4}(1+A^{\frac{1}{1.03D}(n-\frac{1}{2})})^{-4}}
\quad\textrm{by Lemma \ref{ineq}(i)}
\\
&<&3.05A^\frac{1}{8}\prod_{n=1}^\infty
e^{8A^n+A^\frac{n}{D}+A^{\frac{1}{D}(n-1)}+2A^\frac{n}{1.03}+4A^\frac{n}{1.03D}+4A^{\frac{1}{1.03D}(n-\frac{1}{2})}}
\quad\textrm{by Lemma \ref{ineq}(ii)}\\
&=&3.05A^\frac{1}{8}e^{\frac{8A}{1-A}+\frac{A^\frac{1}{D}}{1-A^\frac{1}{D}}+\frac{1}{1-A^\frac{1}{D}}+\frac{2A^\frac{1}{1.03}}{1-A^\frac{1}{1.03}}
+\frac{4A^\frac{1}{1.03D}}{1-A^\frac{1}{1.03D}}
+\frac{4A^\frac{1}{2.06D}}{1-A^\frac{1}{1.03D}}}\\
&<&1\quad\textrm{by the facts $A\leq e^{-\pi\sqrt{39}}$ and
$A^\frac{1}{D}=e^{-\pi\sqrt{3}}$}.
\end{eqnarray*}
This proves the lemma.
\end{proof}

Now we are ready to prove our main theorem of generating ray class
fields.

\begin{theorem}\label{main}
Let $K$ be an imaginary quadratic field with $d_K\leq-39$ and
$N\geq8$. Then
\begin{equation*}
K_{(N)}=K\bigg(x_{(0,~\frac{1}{N})}(\theta),~y_{(0,~\frac{1}{N})}(\theta)^\frac{4}{\gcd(4,~N)}\bigg).
\end{equation*}
In particular, if $4~|~N$ then $K_{(N)}$ is generated by adjoining
to $K$ the $N$-torsion point
\begin{equation}\label{point}
P=\bigg(x_{(0,~\frac{1}{N})}(\theta),~y_{(0,~\frac{1}{N})}(\theta)\bigg)
\end{equation}
of the elliptic curve
\begin{equation}\label{curve}
u(\theta)v(\theta)^3y^2=4x^3-u(\theta)v(\theta)^2x-u(\theta)v(\theta)^4.
\end{equation}
\end{theorem}
\begin{proof}
Since $x_{(0,~\frac{1}{N})}(\tau)\in\mathcal{F}_N$ by definition
(\ref{defF_N}) and
$y_{(0,~\frac{1}{N})}(\tau)^\frac{4}{\gcd(4,~N)}\in\mathcal{F}_N$ by
Proposition \ref{Siegelmodularity}, their singular values
$x_{(0,~\frac{1}{N})}(\theta)$ and
$y_{(0,~\frac{1}{N})}(\theta)^\frac{4}{\gcd(4,~N)}$ lie in $K_{(N)}$
by Proposition \ref{cm}(ii). Assume that any element $(\alpha,~Q)\in
W_{N,~\theta}/\{\pm1_2\}\times \mathrm{C}(d_K)$ fixes both
$x_{(0,~\frac{1}{N})}(\theta)$ and
$y_{(0,~\frac{1}{N})}(\theta)^\frac{4}{\gcd(4,~N)}$. Then we derive
by Propositions \ref{conjugate} and \ref{F_N} that
\begin{eqnarray*}
y_{(0,~\frac{1}{N})}(\theta)^{12N}=\bigg(y_{(0,~\frac{1}{N})}(\theta)^{12N}\bigg)^{(\alpha,~Q)}
=\frac{g_{(0,~\frac{2}{N})\alpha\beta_Q}(\theta_Q)^{12N}}{g_{(0,~\frac{1}{N})\alpha\beta_Q}(\theta_Q)^{48N}}
=\frac{g_{(\frac{2s}{N},~\frac{2t}{N})}(\theta_Q)^{12N}}{g_{(\frac{s}{N},~\frac{t}{N})}(\theta_Q)^{48N}}
\end{eqnarray*}
for some $(s,~t)\in\mathbb{Z}^2$ with $(2s,~2t)\not\in
N\mathbb{Z}^2$. This yields
\begin{equation*}
\bigg|\frac{g_{(0,~\frac{2}{N})}(\theta)}{g_{(0,~\frac{1}{N})}(\theta)^4}\bigg|=
\bigg|\frac{g_{(\frac{2s}{N},~\frac{2t}{N})}(\theta_Q)}{g_{(\frac{s}{N},~\frac{t}{N})}(\theta_Q)^4}\bigg|.
\end{equation*}
Then it follows from Lemma \ref{newlemma} and the conditions
(\ref{reduced}) and (\ref{disc}) for reduced quadratic forms that
\begin{equation*}
Q=\left\{\begin{array}{ll} X^2-\frac{d_K}{4}Y^2 & \textrm{for $d_K\equiv0\pmod{4}$}\\
X^2+XY+\frac{1-d_K}{4}Y^2 & \textrm{for $d_K\equiv1\pmod{4}$},
\end{array}\right.
\end{equation*}
and hence $\beta_Q=1_2$ by (\ref{u1}) and (\ref{u2}), and
$\theta_Q=\theta$ by (\ref{theta_Q}). We then obtain by Propositions
\ref{conjugate} and \ref{rayGalois} that
\begin{equation*}
x_{(0,~\frac{1}{N})}(\theta)=\bigg(x_{(0,~\frac{1}{N})}(\theta)\bigg)^{(\alpha,~Q)}=
x_{(0,~\frac{1}{N})}^{\alpha\beta_Q}(\theta_Q)=
x_{(0,~\frac{1}{N})}^\alpha(\theta)=
\bigg(x_{(0,~\frac{1}{N})}(\theta)\bigg)^\alpha.
\end{equation*}
Hence $\alpha$ should be the identity in $W_{N,~\theta}/\{\pm1_2\}$
because $x_{(0,~\frac{1}{N})}(\theta)$ generates $K_{(N)}$ over $H$
by Proposition \ref{cm}(iii). Therefore $(\alpha,~Q)$ represents the
identity in $\mathrm{Gal}(K_{(N)}/K)$, which proves that the
singular values $x_{(0,~\frac{1}{N})}(\theta)$ and
$y_{(0,~\frac{1}{N})}(\theta)^\frac{4}{\gcd(4,~N)}$ indeed generate
$K_{(N)}$ over $K$.
\par
On the other hand, Proposition \ref{gFourier}(ii) implies that
$u(\theta),~v(\theta)\neq0$, and hence the equation in (\ref{curve})
represents an elliptic curve. And, (\ref{second}) shows that the
point $P$ in (\ref{point}) lies on the elliptic curve as $N$-torsion
point. The proof of the remaining part of the theorem (the case
$4~|~N$) is the same as that of the first part.
\end{proof}

\section{Primitive generators of ray class fields}

In this last section we shall show that some ray class invariants of
imaginary quadratic fields can be constructed from $y$-coordinates
of the elliptic curve in (\ref{curve}) by utilizing the idea of
Schertz(\cite{Schertz}).
\par
Let $K$ be an imaginary quadratic field with discriminant $d_K$ and
$\theta$ as in (\ref{theta}). For a nonzero integral ideal
$\mathfrak{f}$ of $K$ we denote by $\mathrm{Cl}(\mathfrak{f})$ the
ray class group of conductor $\mathfrak{f}$ and write $C_0$ for its
unit class. By definition the ray class field $K_\mathfrak{f}$
modulo $\mathfrak{f}$ of $K$ is a finite abelian extension of $K$
whose Galois group is isomorphic to $\mathrm{Cl}(\mathfrak{f})$ via
the (inverse of) Artin map. If $\mathfrak{f}\neq\mathcal{O}_K$ and
$C\in\mathrm{Cl}(\mathfrak{f})$, then we take an integral ideal
$\mathfrak{c}$ in $C$ so that
$\mathfrak{f}\mathfrak{c}^{-1}=[z_1,~z_2]$ with
$z=\frac{z_1}{z_2}\in\mathfrak{H}$. Now we define the
\textit{Siegel-Ramachandra invariant} by
\begin{equation*}
g_\mathfrak{f}(C)=g_{(\frac{a}{N},~\frac{b}{N})}(z)^{12N}
\end{equation*}
where $N$ is the smallest positive integer in $\mathfrak{f}$ and
$a,~b\in\mathbb{Z}$ such that $1=\frac{a}{N}z_1+\frac{b}{N}z_2$.
This value depends only on the class $C$ and belongs to
$K_\mathfrak{f}$. Furthermore, we have a well-known transformation
formula
\begin{equation}\label{Artin}
g_\mathfrak{f}(C_1)^{\sigma(C_2)}=g_\mathfrak{f}(C_1C_2)
\end{equation}
for $C_1,~C_2\in\mathrm{Cl}(\mathfrak{f})$ where $\sigma$ is the
Artin map(\cite{K-L} Chapter 11 Section 1).
\par
Let $\chi$ be a character of $\mathrm{Cl}(\mathfrak{f})$. We then
denote by $\mathfrak{f}_\chi$ the conductor of $\chi$ and let
$\chi_0$ be the proper character of $\mathrm{Cl}(\mathfrak{f}_\chi)$
corresponding to $\chi$. For a nontrivial character $\chi$ of
$\mathrm{Cl}(\frak{f})$ with $\mathfrak{f}\neq\mathcal{O}_K$ we
define the \textit{Stickelberger element} and the
\textit{$L$-function} as follows:
\begin{eqnarray*} S_\mathfrak{f}(\chi,~g_\mathfrak{f})
&=&\sum_{C\in\mathrm{Cl}(\mathfrak{f})}
\chi(C)\log|g_\mathfrak{f}(C)|\\
L_\mathfrak{f}(s,~\chi)&=&\sum_{\begin{smallmatrix}\mathfrak{a}\neq0~:~\textrm{integral
ideals}\\\gcd(\mathfrak{a},~\mathfrak{f})=\mathcal{O}_K\end{smallmatrix}}\frac{\chi(\mathfrak{a})}{\mathbf{N}_{K/\mathbb{Q}}(\mathfrak{a})^s}\qquad(s\in\mathbb{C}).
\end{eqnarray*}
If $\mathfrak{f}_\chi\neq\mathcal{O}_K$, then we see from the second
Kronecker limit formula that
\begin{equation*}
L_{\mathfrak{f}_\chi}(1,~\chi_0)=T_0S_{\mathfrak{f}_\chi}(\overline{\chi}_0,~g_{\mathfrak{f}_\chi})
\end{equation*}
where $T_0$ is certain nonzero constant depending on
$\chi_0$(\cite{Lang} Chapter 22 Theorem 2). Here we observe that the
value $L_{\mathfrak{f}_\chi}(1,~\chi_0)$ is nonzero(\cite{Janusz}
Chapter IV Proposition 5.7). Moreover, multiplying the above
relation by the Euler factor we derive the identity
\begin{equation}\label{relation}
\prod_{\mathfrak{p}|\mathfrak{f},~
\mathfrak{p}\nmid\mathfrak{f}_\chi}\big(1-\overline{\chi}_0(\mathfrak{p})\big)L_{\mathfrak{f}_\chi}(1,~\chi_0)
=TS_\mathfrak{f}(\overline{\chi},~g_\mathfrak{f})
\end{equation}
where $T$ is certain nonzero constant depending on $\mathfrak{f}$
and $\chi$(\cite{K-L} Chapter 11 Section 2 LF 2).

\begin{lemma}\label{degree}
Let $\mathfrak{f}$ be an integral ideal of $K$. Then we have the
degree formula
\begin{equation*}
[K_{\mathfrak{f}}:K]=\frac{h_K\phi(\mathfrak{f})w(\mathfrak{f})}{w_K}
\end{equation*}
where $h_K$ is the class number of $K$, $\phi$ is the Euler function
for ideals, namely
\begin{equation*}
\phi(\mathfrak{p}^n)=\big(\mathbf{N}_{K/\mathbb{Q}}(\mathfrak{p})-1\big)\mathbf{N}_{K/\mathbb{Q}}(\mathfrak{p})^{n-1}
\end{equation*}
for a power of prime ideal $\mathfrak{p}$ (and we set
$\phi(\mathcal{O}_K)=1$), $w(\mathfrak{f})$ is the number of roots
of unity in $K$ which are $\equiv1\pmod{\mathfrak{f}}$ and $w_K$ is
the number of roots of unity in $K$.
\end{lemma}
\begin{proof}
See \cite{Lang2} Chapter VI Theorem 1.
\end{proof}

\begin{theorem}\label{generator}
Let $\mathfrak{f}\neq\mathcal{O}_K$ be an integral ideal of $K$ with
prime ideal factorization
\begin{equation*}
\mathfrak{f}=\prod_{k=1}^n \mathfrak{p}_k^{e_k}.
\end{equation*}
Assume that
\begin{equation}\label{hypothesis}
[K_\mathfrak{f}:K]>2\sum_{k=1}^n
[K_{\mathfrak{f}\mathfrak{p}_k^{-e_k}}:K].
\end{equation}
Then the singular value
\begin{equation*}
\varepsilon=\frac{g_\mathfrak{f}(C')}{g_\mathfrak{f}(C_0)^4}\quad\textrm{for
any class $C'\in\mathrm{Cl}(\mathfrak{f})$}
\end{equation*}
generates $K_\mathfrak{f}$ over $K$.
\end{theorem}
\begin{proof}
We identify $\mathrm{Cl}(\mathfrak{f})$ with
$\mathrm{Gal}(K_\mathfrak{f}/K)$ via the Artin map. Setting
$F=K(\varepsilon)$ we derive that
\begin{eqnarray}\label{number1}
&&\#\big\{\chi~\mathrm{of}~\mathrm{Gal}(K_\mathfrak{f}/K):~
\chi|_{\mathrm{Gal}(K_\mathfrak{f}/F)}\neq1\big\}\nonumber\\
&=&\#\big\{\chi~\textrm{of}~\mathrm{Gal}(K_\mathfrak{f}/K)\big\}
-\#\big\{\chi~\textrm{of}~\mathrm{Gal}(K_\mathfrak{f}/K)~:~\chi|_{\mathrm{Gal}(K_\mathfrak{f}/F)}=1\big\}\nonumber\\
&=&\#\big\{\chi~\textrm{of}~\mathrm{Gal}(K_\mathfrak{f}/K)\big\}
-\#\big\{\chi~\textrm{of}~\mathrm{Gal}(F/K)\big\}=[K_\mathfrak{f}:K]-[F:K].
\end{eqnarray}
Furthermore, we have
\begin{eqnarray}\label{number2}
&&\#\big\{\chi~\textrm{of}~\mathrm{Gal}(K_\mathfrak{f}/K)~:~\mathfrak{p}_k\nmid\mathfrak{f}_\chi~
\textrm{for
some}~k\big\}\nonumber\\
&=&\#\big\{\chi~\textrm{of}~\mathrm{Gal}(K_\mathfrak{f}/K)~:~\mathfrak{f}_\chi~|~\mathfrak{f}\mathfrak{p}_k^{-e_k}
~\textrm{for some}~k\big\}\nonumber\\
&\leq&\sum_{k=1}^n
\#\big\{\chi~\textrm{of}~\mathrm{Gal}(K_{\mathfrak{f}\mathfrak{p}_k^{-e_k}}/K)\big\}
=\sum_{k=1}^n[K_{\mathfrak{f}\mathfrak{p}_k^{-e_k}}:K] .
\end{eqnarray}
Now, suppose that $F$ is properly contained in $K_\mathfrak{f}$.
Then we get from the hypothesis (\ref{hypothesis}) that
\begin{equation*}
[K_\mathfrak{f}:K]-[F:K]=[K_\mathfrak{f}:K]\bigg(1-\frac{1}{[K_\mathfrak{f}:F]}\bigg)>2\sum_{k=1}^n
[K_{\mathfrak{f}\mathfrak{p}_k^{-e_k}}:K]\bigg(1-\frac{1}{2}\bigg)=\sum_{k=1}^n
[K_{\mathfrak{f}\mathfrak{p}_k^{-e_k}}:K].
\end{equation*}
Thus there exists a character $\psi$ of
$\mathrm{Gal}(K_\mathfrak{f}/K)$ such that
\begin{equation*}
\psi|_{\mathrm{Gal}(K_\mathfrak{f}/F)}\neq1\quad\textrm{and}\quad
\mathfrak{p}_k~|~\mathfrak{f}_\psi~\textrm{for all}~k
\end{equation*}
by (\ref{number1}) and (\ref{number2}). Hence we obtain by
(\ref{relation})
\begin{equation}\label{contradiction}
0\neq
L_{\mathfrak{f}_\psi}(1,~\psi_0)=TS_\mathfrak{f}(\overline{\psi},~g_\mathfrak{f})
\end{equation}
for certain nonzero constant $T$ and the proper character $\psi_0$
of $\mathrm{Cl}(\mathfrak{f}_\psi)$ corresponding to $\psi$. On the
other hand, we achieve that
\begin{eqnarray*}
\big(\psi(C')-4\big)S_\mathfrak{f}(\overline{\psi},~g_\mathfrak{f})
&=&\big(\overline{\psi}(C'^{-1})-4\big)\sum_{C\in\mathrm{Cl}(\mathfrak{f})}\overline{\psi}(C)\log|g_\mathfrak{f}(C)|\\
&=&\sum_{C\in\mathrm{Cl}(\mathfrak{f})}\overline{\psi}(C)\bigg|\bigg(\frac{g_\mathfrak{f}(C')}{g_\mathfrak{f}(C_0)^4}\bigg)^{\sigma(C)}\bigg|\\
&=&\sum_{\begin{smallmatrix}C_1\in\mathrm{Gal}(K_\mathfrak{f}/K)\\C_1\hspace{-0.2cm}
\pmod{\mathrm{Gal}(K_\mathfrak{f}/F)}\end{smallmatrix}}
\sum_{C_2\in\mathrm{Gal}(K_\mathfrak{f}/F)}
\overline{\psi}(C_1C_2)\log|\varepsilon^{\sigma(C_1C_2)}|\\
&=&\sum_{C_1}\overline{\psi}(C_1)\log|\varepsilon^{\sigma(C_1)}|\bigg(\sum_{C_2}\overline{\psi}(C_2)\bigg)
\quad\textrm{by (\ref{Artin}) and the fact}~\varepsilon\in F\\
&=&0\quad\textrm{by the
fact}~\psi|_{\mathrm{Gal}(K_\mathfrak{f}/F)}\neq1,
\end{eqnarray*}
which contradicts (\ref{contradiction}) because $\psi(C')-4\neq0$.
Therefore $F=K_\mathfrak{f}$ as desired.
\end{proof}

\begin{remark}\label{power}
\begin{itemize}
\item[(i)] For any class $C\in\mathrm{Cl}(\mathfrak{f})$, the value $g_\mathfrak{f}(C)$ generates $K_\mathfrak{f}$ over $K$
by the transformation formula (\ref{Artin}) under the assumption
(\ref{hypothesis}).
\item[(ii)]
Any nonzero power of $\varepsilon$ can also generate
$K_\mathfrak{f}$ over $K$ in the proof of Theorem \ref{generator}.
\end{itemize}
\end{remark}

\begin{corollary}\label{Schertz}
Let $N\geq3$ be an odd integer and assume (\ref{hypothesis}) with
$\mathfrak{f}=N\mathcal{O}_K$. Then the singular value
$y_{(0,~\frac{1}{N})}(\theta)^4$ generates $K_{(N)}$ over $K$.
\end{corollary}
\begin{proof}
Observe that for the unit class $C_0$ we have
\begin{equation*}
g_\mathfrak{f}(C_0)=g_{(0,~\frac{1}{N})}(\theta)^{12N}.
\end{equation*}
Since $N$ is odd, $\alpha=\left(\begin{smallmatrix}2 & 0\\0&
2\end{smallmatrix}\right)$ belongs to $W_{N,~\theta}$. Then by
Propositions \ref{rayGalois} and \ref{F_N} we deduce that
\begin{equation*}
g_{(0,~\frac{2}{N})}(\theta)^{12N}=
g_{(0,~\frac{1}{N})\alpha}(\theta)^{12N}=\bigg(g_{(0,~\frac{1}{N})}(\theta)^{12N}\bigg)^\alpha
=g_\mathfrak{f}(C_0)^{\sigma(C')}=g_\mathfrak{f}(C')
\end{equation*}
for some $C'\in\mathrm{Cl}(\mathfrak{f})$. Therefore the singular
value
\begin{equation*}
y_{(0,~\frac{1}{N})}(\theta)^{12N}=\frac{g_{(0,~\frac{2}{N})}(\theta)^{12N}}{g_{(0,~\frac{1}{N})}(\theta)^{48N}}
=\frac{g_\mathfrak{f}(C')}{g_\mathfrak{f}(C_0)^4}
\end{equation*}
generates $K_\mathfrak{f}=K_{(N)}$ over $K$ by Theorem
\ref{generator}. Since $y_{(0,~\frac{1}{N})}(\theta)^4$ belongs to
$K_{(N)}$, it also generates $K_{(N)}$ over $K$.
\end{proof}

\begin{remark}\label{remark} Let $K$ be an imaginary quadratic field
with $d_K\leq-7$ and $N\geq3$ be an odd integer.
\begin{itemize}
\item[(i)] Suppose that $N=p^n(n\geq1)$ where $p$ is an odd prime which is inert or
ramified in $K/\mathbb{Q}$. One can derive by Lemma \ref{degree}
that
\begin{equation*}
[K_{(N)}:K]=\left\{\begin{array}{llll}
\frac{h_K(p^2-1)p^{2(n-1)}}{2}&\geq\frac{h_K(3^2-1)3^{2\cdot0}}{2}&>2h_K
&
\textrm{if $p$ is inert in $K/\mathbb{Q}$}\\
\frac{h_K(p-1)p^{2n-1}}{2}&\geq\frac{h_K(3-1)3^{2\cdot1-1}}{2}&>2h_K
& \textrm{if $p$ is ramified in $K/\mathbb{Q}$}.
\end{array}\right.
\end{equation*}
Thus $\mathfrak{f}=N\mathcal{O}_K$ satisfies the condition
(\ref{hypothesis}) and hence we are able to apply Corollary
\ref{Schertz} for such $N$.
\item[(ii)] Suppose, in general
\begin{equation*}
\mathfrak{f}=N\mathcal{O}_K=\prod_{k=1}^n\mathfrak{p}_k^{e_k}\quad\textrm{with}~n\geq2.
\end{equation*}
Then it follows from Lemma \ref{degree} that the condition
(\ref{hypothesis}) is equivalent to
\begin{equation}\label{assumption}
\frac{1}{2}~>~\sum_{k=1}^n\frac{1}{\phi(\mathfrak{p}_k^{e_k})}.
\end{equation}
Therefore one can also apply Corollary \ref{Schertz} under the
assumption (\ref{assumption}).
\end{itemize}
\end{remark}

\bibliographystyle{amsplain}

\begin{thebibliography}{10}
\bibitem{I-I} N. Ishida and N. Ishii, \emph{The equation for the modular curve $X_1(N)$
derived from the equation for the modular curve $X(N)$}, Tokyo J.
Math. 22 (1999), no. 1, 167-175.
\end{thebibliography}

\end{document}